\documentclass[12pt,russian,a4paper]{article}

\usepackage{amsmath,amssymb,amsthm,xspace,amscd}
\usepackage{amsfonts,amsxtra,latexsym}
\usepackage[cp1251]{inputenc}
\usepackage[russian]{babel}
\usepackage{geometry}

\usepackage{comment}

\usepackage{geometry}
\newtheorem{lemma}{Лемма}[section]
\newtheorem{theorem}{Теорема}[section]
\newtheorem{proposition}{Утверждение}[section]
\newtheorem{corollary}{Следствие}[section]

\newtheorem{remark}{Замечание}[section]

\begin{document}


\begin{center}
{\Large \bf Рост алгебр типа Темперли-Либа с ортогональностью и коммутацией, связанных с 
двухцветными на ребрах графами $K_{n,1}^{(i_1,j_1),\dots,(i_k,j_k)}$}
\end{center}

\begin{center}
 М.В.Заводовский, Ю.С.Самойленко
\end{center}

\section*{Введение}
Изучение структуры наборов ортопроекторов $\{P_k\}_{k=1}^n$, $P_k^2=P_k^*=P_k$ в гильбертовом пространстве $\mathcal H$, связанных различными алгебраическими соотношениями (ортогональность, коммутация, соотношения Темперли-Либа и т.д.) является важной математической задачей. Такие наборы ортопроекторов, как правило, являются представлениями в гильбертовом пространстве $*$-алгебр, порожденных самосопряженными идемпотентами и соответствующими соотношениями. При этом, информация о структуре алгебры полезна при изучении ее $*$-представлений. В частности, полезна информация о размерности или росте алгебр.

В ряде статей изучалась размерность или рост алгебр, заданных образующими и соотношениями, зависящих от графов (неориентированных, ориентированных, с вершиннымСтатьи pdfи или реберными раскрасками и т.д.) и, возможно, еще других параметров (см., например, \cite{gab}-\cite{vms}).

В п.3-4 настоящей статьи изучены размерность и рост (полиномиальный или экспоненциальный) класса алгебр $TL_\tau(K_{n,1}^{(i_1,j_1),\dots,(i_k,j_k)})$ (см. п.2), заданных проекторами и соотношениями, зависящими от двухцветных на ребрах графов или, что тоже самое, графов $K_{n,1}^{(i_1,j_1),\dots,(i_k,j_k)}$ со сплошными и пунктирными ребрами (см. п.1) и вещественного параметра $\tau$, $0<\tau<1$.

Отметим, что $*$-представления алгебры $TL_\tau(K_{n,1}^{(i_1,j_1),\dots,(i_k,j_k)})$ в гильбертовом пространстве задают конфигурации подпространств $S=(\mathcal H; \mathcal H_0,\mathcal H_1,\dots,\mathcal H_n)$ в $\mathcal H$ (см. \cite{ss}) такие, что косинус единственного угла между парами подпространств $\mathcal H_0$ и $\mathcal H_k$, $k=1,\dots,n$ равен $\tau$, а углы между парами подпространств $\mathcal H_k$ и $\mathcal H_j$, $k\neq j$, $k,j=1,\dots,n$ равны либо $90^o$, если вершины $k$ и $j$ не соединены пунктирным ребром и равны $0^o$ или $90^o$, если вершины соединены пунктирным ребром.При этом множество таких неприводимых конфигураций с точностью до унитарной эквивалентности конечно, если алгебра $TL_\tau(K_{n,1}^{(i_1,j_1),\dots,(i_k,j_k)})$ конечномерна и имеет ручную структуру, если алгебра $TL_\tau(K_{n,1}^{(i_1,j_1),\dots,(i_k,j_k)})$ полиномиального роста.

\section{Двухцветные на ребрах графы $K_{n,1}^{(i_1,j_1),\dots,(i_k,j_k)}$ и алгебры $TL_\tau(K_{n,1}^{(i_1,j_1),\dots,(i_k,j_k)})$}

\noindent {\bf 1.1} Двухцветный на ребрах граф $K_{n,1}^{(i_1,j_1),\dots,(i_k,j_k)}$ --- это неориентированная звезда $K_{n,1}$ у которой $n+1$ вершина $VK_{n,1}^{(i_1,j_1),\dots,(i_k,j_k)}=\{0,1,2,\dots,n\}$ и два типа ребер: сплошные ребра $EK_{n,1}^{(i_1,j_1),\dots,(i_k,j_k)}=\{(0,j)\}_{j=1}^n$ и пунктирные ребра $\widetilde E K_{n,1}^{(i_1,j_1),\dots,(i_k,j_k)}=$ $\{(i_1,j_1),\dots,(i_k,j_k)\}$, $i_1,j_1,\dots,$ $i_k,j_k\in$ $\{1,\dots,n\}$, $i_l\neq j_l$, $l=1,\dots,k$. Например, граф $K_{6,1}^{(1,2),(1,3),(1,5),(1,6),(2,6)}$:

\vspace{0.5cm} \setlength{\unitlength}{1mm}
\begin{picture}(80,20)(-10,20)
\linethickness{1pt} \thinlines

\put(35,30){\line(1,0){30}}
\put(50,30){\line(1,1){10}} 
\put(50,30){\line(-1,1){10}}


\put(50,30){\line(1,-1){10}} 
\put(50,30){\line(-1,-1){10}}

\put(35,30){\circle*{1}} 
\put(50,30){\circle*{1}} 
\put(65,30){\circle*{1}} 
\put(60,40){\circle*{1}} 
\put(40,40){\circle*{1}} 
\put(40,20){\circle*{1}} 
\put(60,20){\circle*{1}} 


\put(35,30){\line(1,2){2}} 
\put(40,40){\line(-1,-2){2}}

\put(40,40){\line(1,0){5}} 
\put(47,40){\line(1,0){5}}
\put(54,40){\line(1,0){6}}

\put(35,30){\line(1,-2){2}} 
\put(40,20){\line(-1,2){2}}

\put(35,30){\line(5,2){5}} 
\put(42,33){\line(5,2){5}}
\put(49,36){\line(5,2){5}}
\put(60,40){\line(-5,-2){4}}

\put(35,30){\line(5,-2){5}} 
\put(42,27){\line(5,-2){5}}
\put(49,24){\line(5,-2){5}}
\put(60,20){\line(-5,2){4}}

\put(40,20){\line(0,1){5}} 
\put(40,27){\line(0,1){5}}
\put(40,40){\line(0,-1){6}}

\put(60,40){\line(1,-2){2}} 
\put(65,30){\line(-1,2){2}}

\put(65,30){\line(-1,-2){2}} 
\put(60,20){\line(1,2){2}}

\put(35,27){\makebox(0,0)[a]{$_1$}}
\put(50,27){\makebox(0,0)[a]{$_0$}}
\put(65,27){\makebox(0,0)[a]{$_4$}}
\put(63,40){\makebox(0,0)[a]{$_3$}}
\put(37,40){\makebox(0,0)[a]{$_2$}}
\put(63,20){\makebox(0,0)[a]{$_5$}}
\put(37,20){\makebox(0,0)[a]{$_6$}}
\end{picture}
\vspace{0,5cm}

Граф $K_{m,1}^{(k_1,l_1),\dots,(k_s,l_s)}$ называется подграфом графа $K_{n,1}^{(i_1,j_1),\dots,(i_k,j_k)}$, если $m\leq n$, $s\leq k$, $VK_{m,1}^{(k_1,l_1),\dots,(k_s,l_s)}\subseteq$ $VK_{n,1}^{(i_1,j_1),\dots,(i_k,j_k)}$, $EK_{m,1}^{(k_1,l_1),\dots,(k_s,l_s)}\subseteq$ $EK_{n,1}^{(i_1,j_1),\dots,(i_k,j_k)}$ и \\ $\widetilde EK_{m,1}^{(k_1,l_1),\dots,(k_s,l_s)}\subseteq$ $\widetilde EK_{n,1}^{(i_1,j_1),\dots,(i_k,j_k)}$. 

Два графа $K_{n,1}^{(i_1,j_1),\dots,(i_k,j_k)}$ и $K_{n,1}^{(\tilde i_1,\tilde j_2),\dots,(\tilde i_k,\tilde j_k)}$ изоморфны, если существует перестановка вершин $\{1,\dots,n\}$ устанавливающая взаимнооднозначное соответствие между множеством ребер $\widetilde EK_{n,1}^{(i_1,j_1),\dots,(i_k,j_k)}$ и $\widetilde EK_{n,1}^{(\tilde i_1,\tilde j_2),\dots,(\tilde i_k,\tilde j_k)}$.

Подмножество вершин графа $\{1,\dots,n\}$ разобьем на связанные пунктирными ребрами компоненты $\Gamma_s$ ($s=1,\dots,m$), $k_s$ --- число вершин $s$-той компоненты $\Gamma_s$ ($k_1+\dots+k_m=n$). Обозначим через $\nu(K_{n,1}^{(i_1,j_1),\dots,(i_k,j_k)})$ количество $m$, таких связных по пунктирным ребрам компонент. Например, $\nu(K_{6,1}^{(1,2),(1,3),(1,5),(1,6),(2,6)})=1$, $\nu(K_{4,1}^{(1,2),(3,4)})=2$, $\nu(K_{6,1}^{(1,6),(2,3),(4,5)})=3$.

\noindent {\bf 1.2} Алгебра $TL_\tau(K_{n,1}^{(i_1,j_1),\dots,(i_k,j_k)})$ задается следующими образующими и соотношениями
$$TL_\tau(K_{n,1}^{(i_1,j_1),\dots,(i_k,j_k)})=\mathbb C\Bigl\langle p_0,p_1,\dots,p_n\mid
p_k^2=p_k,\; k=0,1,\dots, n; \; p_ip_0p_i=\tau^2 p_i, $$ $$ p_0p_ip_0=\tau^2 p_0,\;(0<\tau<1,\;i=1,\dots,n);\; p_ip_j=p_jp_i, \; (i,j)\in\{(i_1,j_1),\dots,(i_k,j_k)\};$$ $$ p_ip_j=p_jp_i=0,\; (i,j)\not\in\{(i_1,j_1),\dots,(i_k,j_k)\} \Bigr\rangle.$$

Имеет место следующее утверждение.

\begin{proposition}
Размерность или рост алгебры $TL_\tau(K_{n,1}^{(i_1,j_1),\dots,(i_k,j_k)})$ не зависят от параметра $\tau$.
\end{proposition}
\begin{proof}Рассмотрим граф $K_{n,1}$, в котором все не соединенные ребрами вершины связаны соотношением коммутации в алгебре $TL_\tau(K_{n,1})$. Покажем, что базис Гребнера алгебры $TL_\tau(K_{n,1})$ имеет конечное число элементов. Зададим на образующих алегбры $TL_\tau(K_{n,1})$ следующий порядок $p_1<p_2,\dots,<p_n$. Несложные вычисления дают, что базис Гребнера алегбры $TL_\tau(K_{n,1})$ состоит из элементов $p_i^2$ при $i=0,1,\dots,n$, $p_ip_0p_i-\tau p_i$ и $p_0p_ip_0-\tau p_0$ при $i=0,1,\dots,n$, $p_{i_1}p_{i_2}-p_{i_2}p_{i_1}$ при $i_1,i_2=1,\dots,n$ и $i_1>i_2$, $p_{i_1}p_{i_2}\dots p_{i_m}p_0p_{j_1}p_{j_2}\dots p_{j_s}p_{i_1}-\tau p_{j_1}p_{j_2}\dots p_{j_s}p_{i_2}p_{i_3}\dots p_{i_m}$ при $i_1=1,\dots,n-1$, $i_k=i_{k-1}+1,\dots,n$ $k=2,\dots,m$, $i_{k}<i_{k+1}$ $k=1,\dots,m$, $j_1=1,\dots,n-1$, $j_k=j_{k-1}+1,\dots,n$, $k=2,\dots,s$ $j_{k}<j_{k+1}$, $k=1,\dots,s$. Отсюда видно, что существует не более $n$ старших слов базиса Гребнера алгебры $TL_\tau(K_{n,1})$ максимальной длины и их длина не превышает $n+2$. Но тогда размерность или рост алгебры $TL_\tau(K_{n,1})$ не зависят от параметра $\tau$. Так как алгебра $TL_\tau(K_{n,1}^{(i_1,j_1),\dots,(i_k,j_k)})$ является фактор подалгеброй алгебры $TL_\tau(K_{n,1})$, то ее размерность или рост также не зависят от $\tau$.
\end{proof}

\noindent {\bf 1.3} При удалении вершины $i_0$ графа $K_{n,1}^{(i_1,j_1),\dots,(i_k,j_k)}$, $i_0\in$ $\{1,\dots,n\}$, не связанной с другими вершинами пунктирными ребрами, возникает граф $K_{n-1,1}^{(i_1,j_1),\dots,(i_k,j_k)}$ ($VK_{n-1,1}^{(i_1,j_1),\dots,(i_k,j_k)}=$ $VK_{n,1}^{(i_1,j_1),\dots,(i_k,j_k)}\setminus\{i_0\}$, $EK_{n-1,1}^{(i_1,j_1),\dots,(i_k,j_k)}=$ $EK_{n,1}^{(i_1,j_1),\dots,(i_k,j_k)}\setminus\{(0,i_0)\}$, $\widetilde EK_{n-1,1}^{(i_1,j_1),\dots,(i_k,j_k)}=$ $\widetilde EK_{n,1}^{(i_1,j_1),\dots,(i_k,j_k)}$). Покажем, что алгебры $TL_\tau(K_{n,1}^{(i_1,j_1),\dots,(i_k,j_k)})$ и $TL_\tau(K_{n-1,1}^{(i_1,j_1),\dots,(i_k,j_k)})$ имеют одинаковый рост.

\begin{proposition}
Алгебры $TL_\tau(K_{n,1}^{(i_1,j_1),\dots,(i_k,j_k)})$ и $TL_\tau(K_{n-1,1}^{(i_1,j_1),\dots,(i_k,j_k)})$ имеют одинаковый рост.
\end{proposition}
\begin{proof}
Старшие слова базиса Гребнера алгебры $TL_\tau(K_{n,1}^{(i_1,j_1),\dots,(i_k,j_k)})$ содержат все старшие слова базиса Гребнера алгебры $TL_\tau(K_{n-1,1}^{(i_1,j_1),\dots,(i_k,j_k)})$ и дополнительно элементы $p_np_0p_n$, $p_0p_np_0$, $p_np_i$, $p_ip_n$ для всех $i=1,\dots,n$. Так как алгебра $TL_\tau(K_{n+1,1})$ (алгебра, в которой отсутствуют пунктирные ребра) конечномерна, то эти элементы порождают лишь конечное число элементов линейного базиса. Поэтому добавление к графу $K_{n-1,1}^{(i_1,j_1),\dots,(i_k,j_k)}$ новой вершины $i_n$ и ребра с соотношением ортогональности увеличивает линейный базис алгебры $TL_\tau(K_{n,1}^{(i_1,j_1),\dots,(i_k,j_k)})$ на конечное число элементов. 
\end{proof}

\begin{remark}
Из утверждения 1.2 следует, что для изучения роста алгебр $TL_\tau(K_{n,1}^{(i_1,j_1),\dots,(i_k,j_k)})$ достаточно рассматривать двухцветные графы, у которых каждая из вершин $\{1,\dots,n\}$ участвует в соотношениях коммутации. В дальнейшем мы всегда будем предполагать, что каждая из вершин $\{1,\dots,n\}$ связана с какими-то другими пунктирными ребрами. 
\end{remark}

\noindent {\bf 1.4}
Покажем, что удаление пунктирного ребра не увеличивает рост алгебры.
\begin{proposition}
Пусть  граф $K_{n,1}^{(r_1,s_1),\dots,(r_m,s_m)} \subset K_{n,1}^{(i_1,j_1),\dots,(i_k,j_k)}$, полученный из $K_{n,1}^{(i_1,j_1),\dots,(i_k,j_k)}$ удалением пунктирного ребра. Тогда рост алгебры \\ $TL_\tau(K_{n,1}^{(r_1,s_1),\dots,(r_m,s_m)})$ не превышает рост алгебры $TL_\tau(K_{n,1}^{(i_1,j_1),\dots,(i_k,j_k)})$.
\end{proposition}
\begin{proof}
Предположим, что граф $K_{n,1}^{(r_1,s_1),\dots,(r_m,s_m)} \subset K_{n,1}^{(i_1,j_1),\dots,(i_k,j_k)}$, полученный из $K_{n,1}^{(i_1,j_1),\dots,(i_k,j_k)}$ удалением пунктирного ребра $(i_0,j_0)$. Но тогда линейный базис алгебры $TL_\tau(K_{n,1}^{(r_1,s_1),\dots,(r_m,s_m)})$ содержатся в линейном базисе алгебры $TL_\tau(K_{n,1}^{(i_1,j_1),\dots,(i_k,j_k)})$. Следовательно, рост алгебры $TL_\tau(K_{n,1}^{(r_1,s_1),\dots,(r_m,s_m)})$ не превышает рост алгебры $TL_\tau(K_{n,1}^{(i_1,j_1),\dots,(i_k,j_k)})$.
\end{proof}

\begin{corollary}
Пусть  граф $K_{m,1}^{(i_1,j_1),\dots,(i_l,j_l)} \subset K_{n,1}^{(i_1,j_1),\dots,(i_k,j_k)}$, $m\leq n$, $l\leq k$. Тогда рост алгебры $TL_\tau(K_{m,1}^{(i_1,j_1),\dots,(i_l,j_l)})$ не превышает рост алгебры $TL_\tau(K_{n,1}^{(i_1,j_1),\dots,(i_k,j_k)})$.
\end{corollary}
\begin{proof}
Из утверждений 1.2 и 1.3 следует, что удаление вершин и пунктирных ребер не увеличивает рост алгебры.
\end{proof}

\section{Примеры алгебр $TL_\tau(K_{n,1}^{(i_1,j_1),\dots,(i_k,j_k)})$, их рост}
В данном разделе мы приведем примеры таких графов, что связанные с ними алгебры $TL_\tau(K_{n,1}^{(i_1,j_1),\dots,(i_k,j_k)})$ конечномерны (п. 2.1), бесконечномерны полиномиального роста (п. 2.2) и экспоненциального роста (п. 2.3).

\noindent {\bf 2.1} Приведем примеры конечномерных алгебр $TL_\tau(K_{n,1}^{(i_1,j_1),\dots,(i_k,j_k)})$.

\noindent {\bf 2.1.1} Рассмотрим двухцветный на ребрах граф $K_{n,1}^{(1,2),(1,3),\dots,(1,n)}$:

\vspace{0.5cm} \setlength{\unitlength}{1mm}
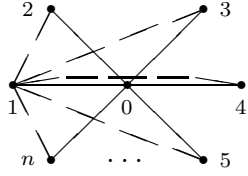
\begin{picture}(80,20)(-10,20)
\linethickness{1pt} \thinlines

\put(35,30){\line(1,0){30}}
\put(50,30){\line(1,1){10}} 
\put(50,30){\line(-1,1){10}}

\put(50,20){\makebox(0,0)[a]{$\dots$}} 

\put(50,30){\line(1,-1){10}} 
\put(50,30){\line(-1,-1){10}}

\put(35,30){\circle*{1}} 
\put(50,30){\circle*{1}} 
\put(65,30){\circle*{1}} 
\put(60,40){\circle*{1}} 
\put(40,40){\circle*{1}} 
\put(40,20){\circle*{1}} 
\put(60,20){\circle*{1}} 

\put(35,30){\line(1,2){2}} 
\put(40,40){\line(-1,-2){2}}

\put(35,30){\line(1,-2){2}} 
\put(40,20){\line(-1,2){2}}

\put(35,30){\line(5,2){5}} 
\put(42,33){\line(5,2){5}}
\put(49,36){\line(5,2){5}}
\put(60,40){\line(-5,-2){4}}

\put(35,30){\line(5,-2){5}} 
\put(42,27){\line(5,-2){5}}
\put(49,24){\line(5,-2){5}}
\put(60,20){\line(-5,2){4}}

\put(35,30){\line(6,1){6}} 
\put(42,31){\line(1,0){4}}
\put(47.5,31){\line(1,0){5}}
\put(58,31){\line(-1,0){4}}
\put(65,30){\line(-6,1){6}}

\put(35,27){\makebox(0,0)[a]{$_1$}}
\put(50,27){\makebox(0,0)[a]{$_0$}}
\put(65,27){\makebox(0,0)[a]{$_4$}}
\put(63,40){\makebox(0,0)[a]{$_3$}}
\put(37,40){\makebox(0,0)[a]{$_2$}}
\put(63,20){\makebox(0,0)[a]{$_5$}}
\put(37,20){\makebox(0,0)[a]{$_n$}}
\end{picture}
\vspace{0,5cm}

\begin{proposition}
Алгебра $TL_\tau(K_{n,1}^{(1,2),(1,3),\dots,(1,n)})$ конечномерна.
\end{proposition}
\begin{proof}
Старшие слова базиса Гребнера алгебры $TL_\tau(K_{n,1}^{(1,2),(1,3),\dots,(1,n)})$ имеют вид $\bigcup_{i=1,\dots,n}\{p_i^2\}$, $\bigcup_{i,j=2,\dots,n}\{p_ip_j\}\setminus\{p_2p_3\}$,
$\bigcup_{i=1,j=2,\dots,n}\{p_ip_jp_i\}$, \\ $\bigcup_{j=1,i=2,\dots,n}\{p_ip_jp_i\}$, $\bigcup_{k=3}^n\{p_kp_1p_2p_k\}$,
$\bigcup_{k=3}^n\{p_2p_kp_1p_k\}$. Данные слова являются запрещенными в линейном базисе алгебры $TL_\tau(K_{n,1}^{(1,2),(1,3),\dots,(1,n)})$.

Известно, что алгебра $TL_\tau(K_{n,1})$ конечномерна и $\dim TL_\tau(K_{n,1})=n^2+1.$ Запрещенные слова алгебры $TL_\tau(K_{n,1})$ содержатся среди запрещенных слов алгебры $TL_\tau(K_{n,1}^{(1,2),(1,3),\dots,(1,n)})$, за исключением слов $p_1p_k$ при $k=2,\dots,n$. Тогда элементы $p_1p_k$ при $k=2,\dots,n$ входят в линейный базис алгебры $TL_\tau(K_{n,1}^{(1,2),(1,3),\dots,(1,n)})$. Покажем, что элементы $p_1p_k$ порождают конечное число новых базисных элементов.

Покажем, что элементы $p_1p_2p_0p_k$ и $p_kp_0p_1p_2$ при любом $k=3,\dots,n$ входят в линейный базис алгебры $TL_\tau(K_{n,1}^{(1,2)})$. Старшие слова базиса Гребнера алгебры $TL_\tau(K_{n,1}^{(1,2)})$ имеют вид $\bigcup_{i=1,\dots,n}\{p_i^2\}$, $\bigcup_{i,j=2,\dots,n}\{p_ip_j\}\setminus\{p_2p_3\}$,
$\bigcup_{i=0,j=1,\dots,n}\{p_ip_jp_i\}$, $\bigcup_{j=0,i=1,\dots,n}\{p_ip_jp_i\}$, $\{p_3p_0p_1p_2\}$,
$\{p_1p_2p_0p_1\}$. Данные слова являются запрещенными в линейном базисе алгебры $TL_\tau(K_{n,1}^{(1,2)})$. Так как ни одно из данных запрещенных слов не встречается в словах $p_1p_2p_0p_k$ и $p_kp_0p_1p_2$ при любом $k=3,\dots,n$, то эти слова являются элементами линейного базиса алгебры $TL_\tau(K_{n,1}^{(1,2)})$.

Среди элементов $p_1p_2p_k$ $(k=0,\dots,n)$ базисным является только элемент $p_1p_2p_0$, а остальные содержат
запрещенные слова. Рассмотрим теперь элементы вида \\ $p_1p_2p_0p_k$ $k=0,\dots,n$. Элементы
$p_1p_2p_0p_k$ при любом $k=3,\dots,n$ входят в линейный базис алгебры $TL_\tau(K_{n,1}^{(1,2)})$. Так как
слова $p_0p_kp_0$ и $p_kp_i$ при любых $k=3,\dots,n$ и $i=1,\dots,n$ являются запрещенными, то новых базисных
элементов не существует.

Среди элементов $p_kp_1p_2$ $(k=0,\dots,n)$ базисным является только элемент $p_0p_1p_2$, а остальные содержат
запрещенные слова. Рассмотрим теперь элементы вида \\ $p_kp_0p_1p_2$ $k=0,\dots,n$. Элементы
$p_kp_0p_1p_2$ при любом $k=3,\dots,n$ входят в линейный базис алгебры $TL_\tau(K_{n,1}^{(1,2)})$. Так как
слова $p_0p_kp_0$ и $p_ip_k$ при любых $k=3,\dots,n$ и $i=1,\dots,n$ являются запрещенными, то новых базисных
элементов не существует.

Аналогично доказывается, что элементы $p_1p_k$ при $k=3,\dots,n$ порождают конечное число базисных элементов алгебры $TL_\tau(K_{n,1}^{(1,2),(1,3),\dots,(1,n)})$. Следовательно, алгебра $TL_\tau(K_{n,1}^{(1,2),(1,3),\dots,(1,n)})$ конечномерна.
\end{proof}

\noindent {\bf 2.1.2} Рассмотрим двухцветный на ребрах граф $K_{3,1}^{(1,2),(1,3),(2,3)}$:

\vspace{0.5cm} \setlength{\unitlength}{1mm}
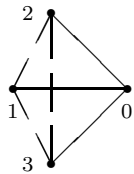
\begin{picture}(80,20)(-10,20)
\linethickness{1pt} \thinlines

\put(35,30){\line(1,0){15}}
\put(50,30){\line(-1,1){10}}
\put(50,30){\line(-1,-1){10}}

\put(35,30){\circle*{1}} 
\put(50,30){\circle*{1}} 
\put(40,40){\circle*{1}} 
\put(40,20){\circle*{1}} 

\put(35,30){\line(1,2){2}} 
\put(40,40){\line(-1,-2){2}}

\put(35,30){\line(1,-2){2}} 
\put(40,20){\line(-1,2){2}}

\put(40,20){\line(0,1){5}} 
\put(40,27){\line(0,1){5}}
\put(40,40){\line(0,-1){6}}

\put(35,27){\makebox(0,0)[a]{$_1$}}
\put(50,27){\makebox(0,0)[a]{$_0$}}
\put(37,40){\makebox(0,0)[a]{$_2$}}
\put(37,20){\makebox(0,0)[a]{$_3$}}
\end{picture}
\vspace{0,5cm}

\begin{proposition}
Алгебра $TL_\tau(K_{3,1}^{(1,2),(1,3),(2,3)})$ конечномерна.
\end{proposition}
\begin{proof}
Так как граф $K_{3,1}$ является подграфом графа Дынкина $D_4$, то алгебра $TL_{3,1}^{(1,2),(1,3)(2,3)}$ конечномерна (см., например, \cite{graham}, \cite{zs}).
\end{proof}

\noindent {\bf 2.2} Приведем примеры алгебр $TL_\tau(K_{n,1}^{(i_1,j_1),\dots,(i_k,j_k)})$ полиномиального роста.

\noindent {\bf 2.2.1} Рассмотрим двухцветный на ребрах граф $K_{4,1}^{(1,2),(3,4)}$:

\vspace{0.5cm} \setlength{\unitlength}{1mm}
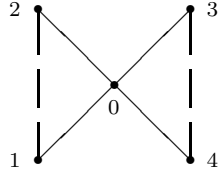
\begin{picture}(80,20)(-10,20)
\linethickness{1pt} \thinlines

\put(50,30){\line(1,1){10}} 
\put(50,30){\line(-1,1){10}}

\put(50,30){\line(1,-1){10}} 
\put(50,30){\line(-1,-1){10}}

\put(50,30){\circle*{1}} 
\put(60,40){\circle*{1}} 
\put(40,40){\circle*{1}} 
\put(40,20){\circle*{1}} 
\put(60,20){\circle*{1}} 

\put(40,20){\line(0,1){5}} 
\put(40,27){\line(0,1){5}}
\put(40,40){\line(0,-1){6}}

\put(60,20){\line(0,1){5}} 
\put(60,27){\line(0,1){5}}
\put(60,40){\line(0,-1){6}}

\put(50,27){\makebox(0,0)[a]{$_0$}}
\put(63,40){\makebox(0,0)[a]{$_3$}}
\put(37,40){\makebox(0,0)[a]{$_2$}}
\put(63,20){\makebox(0,0)[a]{$_4$}}
\put(37,20){\makebox(0,0)[a]{$_1$}}
\end{picture}
\vspace{0,5cm}

\begin{proposition}
Алгебра $TL_\tau(K_{4,1}^{(1,2),(3,4)})$ бесконечномерна и имеет
линейный рост.
\end{proposition}
\begin{proof}
Положим $q_1=p_0p_1p_2$ и $q_2=p_0p_3p_4$. Для любого натурального числа $m$ элемент $(q_1q_2)^m$ принадлежит линейному базису алгебры $TL_\tau(K_{4,1}^{(1,2),(3,4)})$, так как не содержит ни одного запрещенного подслова. Тогда алгебра \\ $TL_\tau(K_{4,1}^{(1,2),(3,4)})$ является бесконечномерной. Алгебра $TL_\tau(K_{4,1}^{(1,2),(3,4)})$ имеет линейный рост, так как является фактор алгеброй алгебры $TL_\tau(K_{4,1}^{(1,2),(1,3),(1,4),(2,3),(3,4)})$, которая бесконечномерна и имеет линейный рост (утверждение 1.4).
\end{proof}

\noindent {\bf 2.2.2} Приведем еще ряд примеров алгебр полиномиального роста.

\begin{proposition}
Все алгебры $TL_\tau(K_{4,1}^{(i_1,j_1),\dots,(i_k,j_k)})$ бесконечномерны и имеют
линейный рост, кроме конечномерной алгебры $TL_\tau(K_{4,1}^{(i_1,i_2),(i_1,i_3),(i_1,i_4)})$.
\end{proposition}
\begin{proof}
Граф $K_{4,1}$ изоморфен графу $\widetilde D_4$, тогда алгебра \\ $TL_\tau(K_{4,1}^{(1,2),(1,3),(1,4),(2,3),(2,4),(3,4)})$ имеет линейный рост (см., например, \cite{graham}, \cite{zs}). По утверждению 2.1 алгебра $TL_\tau(K_{4,1}^{(i_1,i_2),(i_1,i_3),(i_1,i_4)})$ конечномерна, а по утверждению 2.3 алгебра $TL_\tau(K_{4,1}^{(1,2),(3,4)})$ бесконечномерна и имеет линейный рост. Но любой граф $K_{4,1}^{(i_1,j_1),\dots,(i_k,j_k)}$, отличный от $K_{4,1}^{(i_1,i_2),(i_1,i_3),(i_1,i_4)}$ обязательно содержит подграф изоморфный графу $K_{4,1}^{(1,2),(3,4)}$. Тогда в силу утверждения 1.4 алгебра \\ $TL_\tau(K_{4,1}^{(i_1,j_1),\dots,(i_k,j_k)})$, отличная от конечномерной алгебры $TL_\tau(K_{4,1}^{(i_1,i_2),(i_1,i_3),(i_1,i_4)})$, бесконечномерна и имеет линейный рост.
\end{proof}

\noindent {\bf 2.3} Приведем примеры алгебр $TL_\tau(K_{n,1}^{(i_1,j_1),\dots,(i_k,j_k)})$ экспоненциального роста.

\noindent {\bf 2.3.1} Рассмотрим двухцветный на ребрах граф $K_{5,1}^{(1,2),(1,4),(1,5),(2,3)}$:

\vspace{0.5cm} \setlength{\unitlength}{1mm}
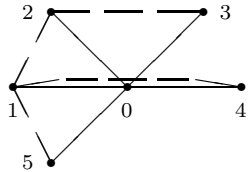
\begin{picture}(80,20)(-10,20)
\linethickness{1pt} \thinlines

\put(35,30){\line(1,0){30}}
\put(50,30){\line(1,1){10}} 
\put(50,30){\line(-1,1){10}}

\put(50,30){\line(-1,-1){10}}

\put(35,30){\circle*{1}} 
\put(50,30){\circle*{1}} 
\put(65,30){\circle*{1}} 
\put(60,40){\circle*{1}} 
\put(40,40){\circle*{1}} 
\put(40,20){\circle*{1}} 

\put(35,30){\line(1,2){2}} 
\put(40,40){\line(-1,-2){2}}

\put(35,30){\line(1,-2){2}} 
\put(40,20){\line(-1,2){2}}

\put(35,30){\line(6,1){6}} 
\put(42,31){\line(1,0){4}}
\put(47.5,31){\line(1,0){5}}
\put(58,31){\line(-1,0){4}}
\put(65,30){\line(-6,1){6}}

\put(40,40){\line(1,0){5}} 
\put(47,40){\line(1,0){5}}
\put(54,40){\line(1,0){6}}

\put(35,27){\makebox(0,0)[a]{$_1$}}
\put(50,27){\makebox(0,0)[a]{$_0$}}
\put(65,27){\makebox(0,0)[a]{$_4$}}
\put(63,40){\makebox(0,0)[a]{$_3$}}
\put(37,40){\makebox(0,0)[a]{$_2$}}

\put(37,20){\makebox(0,0)[a]{$_5$}}
\end{picture}
\vspace{0,5cm}

\begin{proposition}
Алгебра $K_{5,1}^{(1,2),(1,4),(1,5),(2,3)}$ бесконечномерна и имеет
экспоненциальный рост.
\end{proposition}
\begin{proof}
Старшие слова базиса Гребнера алгебры $K_{5,1}^{(1,2),(1,4),(1,5),(2,3)}$ имеют вид $\bigcup_{i=0,\dots,n}\{p_i^2\}$,
$\bigcup_{i,j=1,\dots,n}\{p_ip_j\}\setminus\{\{p_1p_2\}\cup\{p_1p_4\}\cup\{p_1p_5\}\cup\{p_2p_3\}\}$,\\ 
$\bigcup_{i=0,j=1,\dots,n}\{p_ip_jp_i\}$,
$\bigcup_{j=0,i=0,\dots,n}\{p_ip_jp_i\}$, $\{p_3p_0p_1p_3\}$,
$\{p_5p_0p_4p_5\}$, $\{p_2p_0p_1p_2\}$, \\ $\{p_1p_3p_0p_1\}$,
$\{p_4p_5p_0p_2\}$, $\{p_3p_5p_0p_1p_4\}$. Данные слова являются запрещенными в линейном базисе алгебры $K_{5,1}^{(1,2),(1,4),(1,5),(2,3)}$.

Алгебра $K_{5,1}^{(1,2),(1,4),(1,5),(2,3)}$ имеет экспоненциальный
рост, так как её подалгебра, порожденная двумя образующими
$q_1=p_0p_1p_3p_0p_4p_5$ и $q_2=p_0p_5p_0p_1p_3p_0p_1p_5$
свободна (во всевозможных комбинациях элементов $q_1$ и $q_2$ не
содержится ни одного старшего подслова элементов базиса Гребнера
алгебры $K_{5,1}^{(1,2),(1,4),(1,5),(2,3)}$).
\end{proof}

\noindent {\bf 2.3.2} Рассмотрим двухцветный на ребрах граф $K_{5,1}^{(1,2),(2,3),(4,5)}$:

\vspace{0.5cm} \setlength{\unitlength}{1mm}
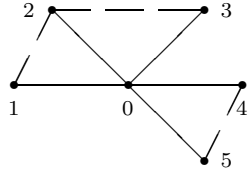
\begin{picture}(80,20)(-10,20)
\linethickness{1pt} \thinlines

\put(35,30){\line(1,0){30}}
\put(50,30){\line(1,1){10}} 
\put(50,30){\line(-1,1){10}}
\put(50,30){\line(1,-1){10}} 

\put(35,30){\circle*{1}} 
\put(50,30){\circle*{1}} 
\put(65,30){\circle*{1}} 
\put(60,40){\circle*{1}} 
\put(40,40){\circle*{1}} 
\put(60,20){\circle*{1}} 

\put(35,30){\line(1,2){2}} 
\put(40,40){\line(-1,-2){2}}

\put(40,40){\line(1,0){5}} 
\put(47,40){\line(1,0){5}}
\put(54,40){\line(1,0){6}}

\put(65,30){\line(-1,-2){2}} 
\put(60,20){\line(1,2){2}}

\put(35,27){\makebox(0,0)[a]{$_1$}}
\put(50,27){\makebox(0,0)[a]{$_0$}}
\put(65,27){\makebox(0,0)[a]{$_4$}}
\put(63,40){\makebox(0,0)[a]{$_3$}}
\put(37,40){\makebox(0,0)[a]{$_2$}}
\put(63,20){\makebox(0,0)[a]{$_5$}}
\end{picture}
\vspace{0,5cm}

\begin{proposition}
Алгебра $TL_\tau(K_{5,1}^{(1,2),(2,3),(4,5)})$ бесконечномерна и имеет
экспоненциальный рост.
\end{proposition}
\begin{proof}
Старшие слова базиса Гребнера алгебры $TL_\tau(K_{5,1}^{(1,2),(2,3),(4,5)})$ имеют вид $\bigcup_{i=0,\dots,n}\{p_i^2\}$, $\bigcup_{i,j=1,\dots,n}\{p_ip_j\}\setminus\{p_1p_2\}$,
$\bigcup_{i=0,j=1,\dots,n}\{p_ip_jp_i\}$, $\bigcup_{j=0,i=1,\dots,n}\{p_ip_jp_i\}$, $\{p_2p_0p_1p_2\}$,
$\{p_1p_2p_0p_1\}$. Данные слова являются запрещенными в линейном базисе алгебры $TL_\tau(K_{5,1}^{(1,2),(2,3),(4,5)})$.

Алгебра $TL_\tau(K_{5,1}^{(1,2),(2,3),(4,5)})$ имеет экспоненциальный
рост, так как её подалгебра, порожденная двумя образующими
$q_1=p_0p_1p_2p_0p_4p_5$ и $q_2=p_0p_2p_3p_0p_4p_5$ свободна (во
всевозможных комбинациях элементов $q_1$ и $q_2$ не содержится ни
одного старшего подслова элементов базиса Гребнера алгебры
$TL_\tau(K_{5,1}^{(1,2),(2,3),(4,5)})$).
\end{proof}

\noindent {\bf 2.3.3} Рассмотрим двухцветный на ребрах граф $K_{6,1}^{(1,6),(2,3),(4,5)}$:

\vspace{0.5cm} \setlength{\unitlength}{1mm}
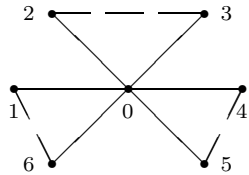
\begin{picture}(80,20)(-10,20)
\linethickness{1pt} \thinlines

\put(35,30){\line(1,0){30}}
\put(50,30){\line(1,1){10}} 
\put(50,30){\line(-1,1){10}}

\put(50,30){\line(1,-1){10}} 
\put(50,30){\line(-1,-1){10}}

\put(35,30){\circle*{1}} 
\put(50,30){\circle*{1}} 
\put(65,30){\circle*{1}} 
\put(60,40){\circle*{1}} 
\put(40,40){\circle*{1}} 
\put(40,20){\circle*{1}} 
\put(60,20){\circle*{1}} 

\put(40,40){\line(1,0){5}} 
\put(47,40){\line(1,0){5}}
\put(54,40){\line(1,0){6}}

\put(35,30){\line(1,-2){2}} 
\put(40,20){\line(-1,2){2}}

\put(65,30){\line(-1,-2){2}} 
\put(60,20){\line(1,2){2}}

\put(35,27){\makebox(0,0)[a]{$_1$}}
\put(50,27){\makebox(0,0)[a]{$_0$}}
\put(65,27){\makebox(0,0)[a]{$_4$}}
\put(63,40){\makebox(0,0)[a]{$_3$}}
\put(37,40){\makebox(0,0)[a]{$_2$}}
\put(63,20){\makebox(0,0)[a]{$_5$}}
\put(37,20){\makebox(0,0)[a]{$_6$}}
\end{picture}
\vspace{0,5cm}

\begin{proposition}
Алгебра $TL_\tau(K_{6,1}^{(1,6),(2,3),(4,5)})$ бесконечномерна и имеет
экспоненциальный рост.
\end{proposition}
\begin{proof}
Старшие слова базиса Гребнера алгебры $TL_\tau(K_{6,1}^{(1,6),(2,3),(4,5)})$ имеют вид $\bigcup_{i=0,\dots,n}\{p_i^2\}$,
$\bigcup_{i,j=1,\dots,n}\{p_ip_j\}\setminus\{\{p_1p_6\}\cup\{p_2p_3\}\cup\{p_4p_5\}\}$,
$\bigcup_{i=0,j=1,\dots,n}\{p_ip_jp_i\}$,
$\bigcup_{j=0,i=0,\dots,n}\{p_ip_jp_i\}$, $\{p_3p_0p_1p_3\}$,
$\{p_5p_0p_4p_5\}$, $\{p_2p_0p_6p_2\}$, $\{p_1p_3p_0p_1\}$,
$\{p_4p_5p_0p_2\}$, \\ $\{p_6p_2p_0p_6\}$. Данные слова являются запрещенными в линейном базисе алгебры \\ $TL_\tau(K_{6,1}^{(1,6),(2,3),(4,5)})$.

Алгебра $TL_\tau(K_{6,1}^{(1,6),(2,3),(4,5)})$ имеет экспоненциальный
рост, так как её подалгебра, порожденная двумя образующими
$q_1=p_0p_1p_3p_0p_4p_5$ и $q_2=p_0p_6p_2p_0p_1p_3p_0p_4p_5$
свободна (во всевозможных комбинациях элементов $q_1$ и $q_2$ не
содержится ни одного старшего подслова элементов базиса Гребнера
алгебры $TL_\tau(K_{6,1}^{(1,6),(2,3),(4,5)})$).
\end{proof}

\section{Рост алгебры $TL_\tau(K_{n,1}^{(i_1,j_1),\dots,(i_k,j_k)})$ и число \\ $\nu(K_{n,1}^{(i_1,j_1),\dots,(i_k,j_k)})$}

\noindent {\bf 3.1} Имеет место следующее утверждение.
\begin{proposition}
Если алгебра $TL_\tau(K_{n,1}^{(i_1,j_1),\dots,(i_k,j_k)})$ конечномерна, то \\ $\nu(K_{n,1}^{(i_1,j_1),\dots,(i_k,j_k)})=1$.
\end{proposition}
\begin{proof}
Предположим, что $\nu(K_{n,1}^{(i_1,j_1),\dots,(i_k,j_k)})>1$, тогда граф $K_{n,1}^{(i_1,j_1),\dots,(i_k,j_k)}$ обязательно содержит подграф изоморфный графу $K_{4,1}^{(1,2),(3,4)}$. Так как алгебра $TL_\tau(K_{4,1}^{(1,2),(3,4)})$ бесконечномерна, то $\nu(K_{n,1}^{(i_1,j_1),\dots,(i_k,j_k)})=1$.
\end{proof}

\begin{remark}
Обратное утверждение не имеет места. Если \\ $\nu(K_{n,1}^{(i_1,j_1),\dots,(i_k,j_k)})=1$, то в зависимости от графа $K_{n,1}^{(i_1,j_1),\dots,(i_k,j_k)}$ алгебра \\ $TL_\tau(K_{n,1}^{(i_1,j_1),\dots,(i_k,j_k)})$ может быть как конечномерной, так и бесконечномерной полиномиального или экспоненциального роста (см. примеры в п.2).
\end{remark}

\noindent {\bf 3.2} Имеет место следующее утверждение.
\begin{proposition}
Если алгебра $TL_\tau(K_{n,1}^{(i_1,j_1),\dots,(i_k,j_k)})$ бесконечномерна и имеет полиномиальный рост, то $\nu(K_{n,1}^{(i_1,j_1),\dots,(i_k,j_k)})\leq 2$.
\end{proposition}
\begin{proof}
Предположим, что $\nu(K_{n,1}^{(i_1,j_1),\dots,(i_k,j_k)})>2$, тогда граф $K_{n,1}^{(i_1,j_1),\dots,(i_k,j_k)}$ обязательно содержит подграф изоморфный графу $K_{6,1}^{(1,6),(2,3),(4,5)}$. Так как алгебра $TL_\tau(K_{6,1}^{(1,6),(2,3),(4,5)})$ бесконечномерна и имеет экспоненциальный рост, то \\ $\nu(K_{n,1}^{(i_1,j_1),\dots,(i_k,j_k)})\leq 2$.
\end{proof}

\begin{remark}
Обратное утверждение не имеет места. Если \\ $\nu(K_{n,1}^{(i_1,j_1),\dots,(i_k,j_k)})=2$, то в зависимости от графа $K_{n,1}^{(i_1,j_1),\dots,(i_k,j_k)}$ алгебра \\ $TL_\tau(K_{n,1}^{(i_1,j_1),\dots,(i_k,j_k)})$ может быть только бесконечномерной полиномиального или экспоненциального роста (см. примеры в п.2).
\end{remark}

\noindent {\bf 3.3} Имеет место следующее утверждение.
\begin{proposition}
Если $\nu(K_{n,1}^{(i_1,j_1),\dots,(i_k,j_k)})\geq 3$, то алгебра $TL_\tau(K_{n,1}^{(i_1,j_1),\dots,(i_k,j_k)})$ всегда является бесконечномерной экспоненциального роста.
\end{proposition}
\begin{proof}
Если $\nu(K_{n,1}^{(i_1,j_1),\dots,(i_k,j_k)})\geq 3$, то она обязательно содержит подграф изоморфный графу $K_{6,1}^{(1,6),(2,3),(4,5)}$, для которого соответствующая алгебра бесконечномерна и имеет экспоненциальный рост.
\end{proof}

\begin{remark}
Обратное утверждение не имеет места: у бесконечномерных алгебр $TL_\tau(K_{n,1}^{(i_1,j_1),\dots,(i_k,j_k)})$ экспоненциального роста $\nu(K_{n,1}^{(i_1,j_1),\dots,(i_k,j_k)})$ может быть равно 1 или 2 (см. примеры в п.2).
\end{remark}

\section{Основная теорема}
Напомним, что в п.2-4 статьи предполагается, что каждая вершина $i\in\{1,\dots,n\}$ графа $K_{n,1}^{(i_1,j_1),\dots,(i_k,j_k)}$ соединена хотя бы с одной другой вершиной пунктирными ребрами (см. замечание 1.1).

Перейдем к основной теореме работы.

\noindent {\bf Теорема} Рассмотрим алгебру $TL_\tau(K_{n,1}^{(i_1,j_1),\dots,(i_k,j_k)})$
\begin{itemize}
\item[(i)] лгебра $TL_\tau(K_{n,1}^{(i_1,j_1),\dots,(i_k,j_k)})$ конечномерна тогда и только тогда, когда граф $K_{n,1}^{(i_1,j_1),\dots,(i_k,j_k)}$ изоморфен графу $K_{n,1}^{(1,2),(1,3),\dots,(1,n)}$ или $K_{3,1}^{(1,2),(1,3),(2,3)}$;

\item[(ii)] алгебра $TL_\tau(K_{n,1}^{(i_1,j_1),\dots,(i_k,j_k)})$ бесконечномерна и имеет линейный рост тогда и только тогда, когда $n=4$ и граф $K_{4,1}^{(i_1,j_1),\dots,(i_k,j_k)}$ не изоморфен графу $K_{4,1}^{(1,2),(1,3),(1,4)}$, для которго соответствующая алгебра конечномерна;

\item[(iii)] для всех графов, отличных от графов п. (i) и (ii), алгебра $TL_\tau(K_{n,1}^{(i_1,j_1),\dots,(i_k,j_k)})$ бесконечномерна и имеет экспоненциальный рост; любой такой граф \\ $K_{n,1}^{(i_1,j_1),\dots,(i_k,j_k)}$ содержит подграф $K_{5,1}^{(1,2),(2,3),(4,5)}$, $K_{6,1}^{(1,6),(2,3),(4,5)}$ или \\ $K_{5,1}^{(1,2),(1,4),(1,5),(2,3)}$.
\end{itemize}

\begin{proof}
(i) В пункте 2.1 было доказано, что если граф $K_{n,1}^{(i_1,j_1),\dots,(i_k,j_k)}$ изоморфен графу $K_{n,1}^{(1,2),(1,3),\dots,(1,n)}$ или $K_{3,1}^{(1,2),(1,3),(2,3)}$, то алгебра $TL_\tau(K_{n,1}^{(i_1,j_1),\dots,(i_k,j_k)})$ конечномерна. 

Обратно, предположим, что алгебра $TL_\tau(K_{n,1}^{(i_1,j_1),\dots,(i_k,j_k)})$ конечномерна, тогда по утверждению 3.1 $\nu(K_{n,1}^{(i_1,j_1),\dots,(i_k,j_k)})=1$. 

При $n=3$ алгебра $TL_\tau(K_{3,1}^{(1,2),(1,3),(2,3)})$ конечномерна (см., например, \cite{graham}) и любой подграф графа $K_{3,1}^{(1,2),(1,3),(2,3)}$, в котором каждая из вершин $\{1,2,3\}$ участвует в соотношениях коммутации, изоморфен графу $K_{3,1}^{(1,2),(1,3)}$ или $K_{3,1}^{(1,2),(1,3),(2,3)}$. 

Пусть $n>3$ и граф $K_{n,1}^{(i_1,j_1),\dots,(i_k,j_k)}$ содержит более одной вершины, из которой выходят два или больше пунктирных ребра. Тогда граф $K_{n,1}^{(i_1,j_1),\dots,(i_k,j_k)}$ обязательно содержит подграф изоморфный графу  $K_{n,1}^{(1,2),(3,4)}$, для которого соответствующая алгебра бесконечномерна и имеет линейный рост. Следовательно, в графе $K_{n,1}^{(i_1,j_1),\dots,(i_k,j_k)}$ не существует более одной вершины, из которой выходит несколько ребер. Единственный граф, удовлетворяющий этому свойству, является граф изоморфный графу $K_{n,1}^{(1,2),(1,3),\dots,(1,n)}$.

(ii) В пункте 2.2 было доказано, что все алгебры $TL_\tau(K_{4,1}^{(i_1,j_1),\dots,(i_k,j_k)})$ бесконечномерны и имеют
линейный рост, кроме конечномерной алгебры \\ $TL_\tau(K_{4,1}^{(i_1,i_2),(i_1,i_3),(i_1,i_4)})$.

Обратно, предположим, что алгебра $TL_\tau(K_{n,1}^{(i_1,j_1),\dots,(i_k,j_k)})$ бесконечномерна и имеет полиномиальный рост. По утверждению 3.2 $\nu(K_{n,1}^{(i_1,j_1),\dots,(i_{k_1},j_{k_1})})\leq 2$. 

Если $\nu(K_{n,1}^{(i_1,j_1),\dots,(i_{k_1},j_{k_1})})= 2$, то при $n\geq 5$ не существует графа $K_{n,1}^{(i_1,j_1),\dots,(i_{k_1},j_{k_1})}$ с двумя компонентами связности, который бы не содержал в качестве подграфа граф $K_{5,1}^{(1,2),(2,3),(4,5)}$, для которого соответствующая алгебра $TL_\tau(K_{5,1}^{(1,2),(2,3),(4,5)})$ бесконечномерна и имеет экспоненциальный рост. 

Пусть $\nu(K_{n,1}^{(i_1,j_1),\dots,(i_{k_1},j_{k_1})})= 1$ и $n=5$, тогда граф $K_{n,1}^{(i_1,j_1),\dots,(i_k,j_k)}$ содержит не менее четырех пунктирных ребер и, следовательно, содержит в качестве подграфа один из графов изоморфных $K_{5,1}^{(1,2),(1,3),(1,4),(1,5)}$, $K_{5,1}^{(1,2),(2,3),(3,4),(4,5)}$ или \\ $K_{5,1}^{(1,2),(1,4),(1,5),(2,3)}$. Алгебры $TL_\tau(K_{5,1}^{(1,2),(2,3),(3,4),(4,5)})$ и $TL_\tau(K_{5,1}^{(1,2),(1,4),(1,5),(2,3)})$ бесконечномерны и имеют экспоненциальный рост. Алгебра $TL_\tau(K_{5,1}^{(1,2),(1,3),(1,4),(1,5)})$ конечномерна и добавление пунктирного ребра $(i_0,j_0)$ в любом месте приведет к тому, что граф $K_{5,1}^{(1,2),(1,3),(1,4),(1,5),(i_0,j_0)}$ будет содержать подграф изоморфный \\ $K_{5,1}^{(1,2),(1,4),(1,5),(2,3)}$, для которого сооответствующая алгебра бесконечномерна и имеет экспоненциальный рост.

Если $n>5$, то граф $K_{n,1}^{(i_1,j_1),\dots,(i_k,j_k)}$ также будет содержать в качестве подграфа один из графов изоморфных $K_{n,1}^{(1,2),(1,3),\dots,(1,n)}$, $K_{5,1}^{(1,2),(2,3),(3,4),(4,5)}$ или $K_{5,1}^{(1,2),(1,4),(1,5),(2,3)}$. Алгебра $TL_\tau(K_{n,1}^{(1,2),(1,3),\dots,(1,n)})$ конечномерна и добавление пунктирного ребра $(i_0,j_0)$ в любом месте приведет к тому, что граф $K_{n,1}^{(1,2),(1,3),\dots,(1,n),(i_0,j_0)}$ будет содержать подграф изоморфный $K_{5,1}^{(1,2),(1,4),(1,5),(2,3)}$, для которого сооответствующая алгебра бесконечномерна и имеет экспоненциальный рост. Следовательно, при \\ $\nu(K_{n,1}^{(i_1,j_1),\dots,(i_{k_1},j_{k_1})})= 1$ и $n\geq 4$ алгебра $TL_\tau(K_{n,1}^{(i_1,j_1),\dots,(i_k,j_k)})$ не может иметь полиномиальный рост. Тогда и граф $K_{n,1}^{(i_1,j_1),\dots,(i_k,j_k)}$ изоморфен одному из графов $K_{4,1}^{(i_1,j_1),\dots,(i_k,j_k)}$ за исключением графа $K_{4,1}^{(1,2),(1,3),(1,4)}$, для которого соответствующая алгебра конечномерна.

(iii) Если алгебра $TL_\tau(K_{n,1}^{(i_1,j_1),\dots,(i_k,j_k)})$ отлична от графов из пунктов (i) и (ii), то она содержит один из подграфов изоморфных $K_{5,1}^{(1,2),(2,3),(4,5)}$, $K_{6,1}^{(1,6),(2,3),(4,5)}$ или $K_{5,1}^{(1,2),(1,4),(1,5),(2,3)}$, следовательно, она бесконечномерна и имеет экспоненциальный рост.

По утверждению 3.3 при $\nu(K_{n,1}^{(i_1,j_1),\dots,(i_{k_1},j_{k_1})})\geq 3$ алгебра $TL_\tau(K_{n,1}^{(i_1,j_1),\dots,(i_k,j_k)})$ всегда бесконечномерна и имеет экспоненциальный рост. Если $\nu(K_{n,1}^{(i_1,j_1),\dots,(i_{k_1},j_{k_1})})= 2$ и $n\geq 5$, то граф $K_{n,1}^{(i_1,j_1),\dots,(i_k,j_k)}$ обязательно содержит подграф изоморфный $K_{5,1}^{(1,2),(2,3),(4,5)}$, для которого соответсвующая алгебра бесконечномерна и имеет экспоненциальный рост. Если $\nu(K_{n,1}^{(i_1,j_1),\dots,(i_{k_1},j_{k_1})})=1$, то $n\geq 5$. При $n\geq 5$ и \\ $\nu(K_{n,1}^{(i_1,j_1),\dots,(i_{k_1},j_{k_1})})= 1$ граф $K_{n,1}^{(i_1,j_1),\dots,(i_k,j_k)}$ содержит подграф изоморфный \\ $K_{5,1}^{(1,2),(2,3),(4,5)}$ или $K_{5,1}^{(1,2),(1,4),(1,5),(2,3)}$.

В пункте 2.3 было доказано, что если граф $K_{n,1}^{(i_1,j_1),\dots,(i_k,j_k)}$ изоморфен графу $K_{5,1}^{(1,2),(2,3),(4,5)}$, $K_{6,1}^{(1,6),(2,3),(4,5)}$ или $K_{5,1}^{(1,2),(1,4),(1,5),(2,3)}$, то алгебра $TL_\tau(K_{n,1}^{(i_1,j_1),\dots,(i_k,j_k)})$ бесконечномерна и имеет экспоненциальный рост. 
\end{proof}

\end{document}